\documentclass{article}
\usepackage[utf8]{inputenc}
\usepackage{amsmath}
\usepackage{etoolbox}
\usepackage[usenames]{color}
\usepackage{amsfonts}
\usepackage{mathrsfs}
\usepackage{amsthm}
\usepackage{float}
\usepackage{lmodern}
\usepackage[table]{xcolor}
\usepackage{graphicx,array,multirow,pict2e}
\usepackage{booktabs}
\usepackage{siunitx}
\usepackage{diagbox} %table split headers
\usepackage{rotating}
\usepackage{eqparbox}
\usepackage{makecell,caption}
\usepackage[subpreambles=true]{standalone}

\patchcmd{\subequations}{}%
{}{}{}

\makeatletter
\let\orgdescriptionlabel\descriptionlabel
\renewcommand*{\descriptionlabel}[1]{%
  \let\orglabel\label
  \let\label\@gobble
  \phantomsection
  \edef\@currentlabel{#1}%
  \let\label\orglabel
  \orgdescriptionlabel{#1}%
}
\makeatother

\DeclareMathOperator{\spn}{span}

\newcommand{\abs}[1]{\left\lvert #1\right\rvert}

\newcommand{\OV}{\operatorname{\mathsf{V}}}

\newcommand{\OK}{\operatorname{\mathsf{K}}}
\newcommand{\OW}{\operatorname{\mathsf{W}}}

\newcommand{\Id}{\operatorname{\mathsf{Id}}}

\newcommand{\VV}{\mathbb{V}}
\newcommand{\spV}{\mathbf{V}}

\newcommand{\calH}{\mathcal{H}}

\newcommand{\calS}{\mathcal{S}}

\renewcommand{\O}{\mathcal{O}}

\newcommand{\tN}{\text{N}}

\newcommand{\pw}{\text{pw}}

\newcommand{\Vx}{\mathbf{x}}

\newcommand{\Vn}{\mathbf{n}}

\newcommand{\R}{\mathbb{R}}

\newcommand{\N}{\mathbb{N}}

\newcommand{\dual}[2]{\left\langle #1\,,\,#2\right\rangle}
\newcommand{\inner}[2]{\left( #1\,,\,#2\right)}
\newcommand{\norm}[2]{\left\lVert #1\right\rVert_{#2}}
\definecolor{orange}{rgb}{1,0.4,0}
\definecolor{green}{rgb}{0,0.4,0}

\newcommand{\sgn}[1]{\text{sign}\left( #1 \right)}

\newtheorem{theorem}{Theorem}[section]
\newtheorem{lemma}[theorem]{Lemma}
\newtheorem{proposition}[theorem]{Proposition}
\newtheorem{corollary}[theorem]{Corollary}
\newtheorem{remark}[theorem]{Remark}
\newtheorem{definition}[theorem]{Definition}
\newtheorem{assumption}[theorem]{Assumption}

\title{Towards Stable Second-Kind Boundary Integral Equations for Transient Wave Problems}
\author{D.~Hoonhout$^1$, C.~Urz\'ua-Torres$^1$\thanks{The research leading to this 
publication received funding from the Dutch Research 
Council (NWO) under the NWO-Talent Programme Veni with the project number 
VI.Veni.212.253.}}
\date{
{\small $^1$Delft Institute of Applied Mathematics, TU Delft, \\
Mekelweg 4, 2628CD Delft, The Netherlands }\\[3mm]
}

\begin{document}
\maketitle

%%%%%%%%%%%%%%%%%%%%%%%%%%%%%%%%%%%%%%%%%%%%%%%%%%%%%%%%%%%%%%%%%%%%%%%%%%%%%%%
\begin{abstract}
  In this paper, we discuss the stable discretisation of the double layer 
  boundary integral operator for the wave equation in $1d$. For this, we show that the 
  boundary integral formulation is $L^2$-elliptic and also inf-sup stable in 
  standard energy spaces. This turns out to be a particular case of a recent 
  result on the inf-sup stability of boundary integral operators for the wave 
  equation and contributes to its further understanding. 
  Moreover, we present the first BEM discretisations of second-kind operators for 
  the wave equation for which stability is guaranteed and a complete numerical 
  analysis is offered. We validate our theoretical findings with numerical 
  experiments. 
\end{abstract}
%%%%%%%%%%%%%%%%%%%%%%%%%%%%%%%%%%%%%%%%%%%%%%%%%%%%%%%%%%%%%%%%%%%%%%%%%%%%%%%

% % REQUIRED
% \begin{keywords}
%   Transient wave equation, retarded potentials, integral equations, coercivity
% \end{keywords}
% 
% % REQUIRED
% \begin{AMS}
%    	35L05, 45P05, 47G10
% \end{AMS}

%%%%%%%%%%%%%%%%%%%%%%%%%%%%%%%%%%%%%%%%%%%%%%%%%%%%%%%%%%%%%%%%%%%%%%%%%%%%%%%
\section{Introduction}
\label{sec:intro}

Time-domain boundary integral equations (BIEs) and boundary element methods 
(BEM) for the transient wave equation are well-established in the literature. 
We refer 
to \cite{ADG09, ADG08, BHD86, CoS17, GNS17, HQS17, JoR17, PoS21, Say13, Say16} 
and the references therein, to mention a few. However, in spite of their broad 
use and success, its numerical analysis is still incomplete, as pointed out in 
\cite{CoS17, JoR17}. 

Most of the current literature is based on the groundbreaking work of Bamberger 
and Ha-Duong \cite{BHD86}, where they proposed a time-dependent variational 
formulation for transient wave problems for which existence and uniqueness is 
shown using the Laplace transform. Their results established ellipticity and 
continuity in different weighted Sobolev spaces. 
% Although the ellipticity results are sufficient to ensure unconditional 
% stability properties for conforming Galerkin approximations, 
This so-called \emph{norm gap} prevents the use of standard numerical analysis 
tools such as the Lax-Milgram theorem.

Later, Aimi et al. introduced a different formulation \cite{ADG08}, usually 
referred to as \emph{Energetic BEM}, which is elliptic and continuous in $L^2$ 
for one dimensional problems, thus providing a complete error and stability 
analysis in $1d$. Numerical evidence also supports their applicability in $2d$. 
Nonetheless, to the best of the authors' knowledge, no numerical analysis of 
this method is available for higher dimensions. In addition, energetic BEM for 
Neumann problems requires higher regularity on the ansatz space than the 
expected one (which is a drawback for overly perfectionist mathematicians).

More recently, Steinbach and Urz\'ua-Torres developed a new space-time 
formulation \cite{StU22} that does not suffer from a norm-gap and for which 
existence and uniqueness is proven with the expected regularity in all 
dimensions. In order to achieve this, they had to work in a more general 
framework. Instead of ellipticity, they showed inf-sup stability of the BIEs in 
a new family of trace spaces, which preserve the expected regularity in space 
and time. Furthermore, in joint work with Zank \cite{SUZ21}, 
they regularised their formulation by means of the \emph{Modified Hilbert 
transform}. The resulting variational form is proven to be elliptic in the 
natural energy spaces and can also be seen as a generalisation of the energetic 
BEM approach. This last result is currently only available in $1d$.

There is still a lot to comprehend about the new framework introduced by
Steinbach and Urz\'ua-Torres. In this paper, we aim to understand this further
and work with the second-kind formulation from \cite{StU22}. Moreover, as a 
first step in this direction, we consider only one spatial variable and
propose two stable space-time Galerkin discretisations for the second-kind 
BIE that corresponds to the transient wave equation with prescribed Dirichlet 
data in $1d$. For this, we develop the analysis both in $L^2$, and in the 
natural trace space $H^{1/2}_{0,}(\Sigma)$ and its dual space $[H^{1/2}_{0,}(
\Sigma)]^\prime$. Moreover, we show that the corresponding second-kind 
formulations have a unique solution when considering these spaces. Finally, we 
prove that the proposed discretisation is stable under certain mesh conditions, 
that agree with the ones found in \cite{HLS23}.

The introduced approach distinguishes itself from existing methods in $1d$ in a variety 
of ways: (i) The regularity assumption on the Dirichlet data is weakened, when 
compared with energetic BEM; (ii) It does \emph{not} require stabilisation by 
introducing operators such as the Modified Hilbert transform, and thus has a 
simpler implementation than \cite{SUZ21}; (iii) Numerical results show that it 
%It 
works for all meshes where the $L^2$-orthogonal projection is $H^1$-stable, and 
thus retains the potential use of space-time adaptive algorithms and 
parallelisation. 

The novelty and appeal of this approach relies on its simplicity and on the 
fact that it offers the first standalone stable second-kind formulation for the 
wave equation, which benefits from the well-conditioning of this type of BIEs. 
Although this paper focuses on $1d$
and the analysis we provide cannot be extended directly, the message we hope 
to convey is that there is no reason to believe that second kind formulations 
for the wave equation are hopeless. Moreover, they may actually offer some 
computational advantages.

Indeed, the literature regarding BEM for transient wave problems mainly focuses
on \emph{first-kind} formulations and its discretisation. The review article by 
Costabel and Sayas \cite{CoS17} gives an extensive overview of the currently 
available approaches. Examples of these works include 
\cite{GNS17,GOS19,GOS20,GOS18,PoS21,Say13}. 
Some of the few cases that examine \emph{second-kind} formulations are 
\cite{HaD87}, and \cite{EGH16, BGH20} where combined field approaches are 
required for stability.

This paper is structured as follows: In Section~\ref{sec:prelim}, we briefly 
introduce the required mathematical framework. Then we show in 
Section~\ref{sec:2ndKF} that the bilinear forms for the second-kind variational 
formulation proposed in \cite{StU22} are isomorphisms both in the natural energy 
space and even elliptic in $L^2$ when restricted to one spatial-dimension. 
Section~\ref{sec:dis} is dedicated to present the two related discretisations, 
namely, using piecewise constant basis functions as test and trial spaces, and using 
piecewise linear basis functions with zero initial condition. We complete the 
Section by showing discrete inf-sup stability for the latter. In Section~\ref{sec:EE}, 
we proceed to derive the corresponding error estimates.
Finally, we provide numerical experiments in Section~\ref{sec:numexp} to 
validate our theoretical claims.

%%%%%%%%%%%%%%%%%%%%%%%%%%%%%%%%%%%%%%%%%%%%%%%%%%%%%%%%%%%%%%%%%%%%%%%%%%%%%%%
\section{Preliminaries}
\label{sec:prelim}

%%%%%%%%%%%%%%%%%%%%%%%%%%%%%%%%%%%%%%%%%%%%%%%%%%%%%%%%%%%
\subsection{Model problem}
\label{ssec:modelprob}
Let $\Omega$ be an interval $\Omega = (0,L)\subset \R$, with boundary $\Gamma = 
\lbrace 0 , L \rbrace$. For a finite time interval $(0,T)$, 
we define the interior \emph{space-time cylinder} $Q := \Omega \times (0,T) \, 
\subset \, \R^{2}$, and its lateral boundary $\Sigma := \Gamma \times [0,T]$. 
In an analogous fashion, the exterior space-time cylinder is denoted by 
$Q_{ext}:= (\R\setminus\overline{\Omega})\times (0, T )$.

We denote the D'Alembert operator by $\Box := \partial_{tt} - \Delta_x$,
and write the \emph{interior Dirichlet initial boundary value problem
for the wave equation} as
\begin{equation}
\label{eq:IBVP}
\begin{array}{rclcl}
\Box u(x,t) & = & f(x,t) & & \text{for} \; (x,t) \in Q,  \\
u(x,t) & = & g(x,t) & & \text{for} \; (x,t) \in \Sigma,\\
u(x,0) = \partial_t u(x,t)_{|t=0} & = & 0 & & \text{for} \; x \in \Omega,
\end{array}
\end{equation}
and the exterior Dirichlet initial boundary value problem for the wave equation
\begin{equation}
\label{eq:IBVP_exterior}
\begin{array}{rclcl}
\Box u(x,t) & = & f(x,t) & & \text{for} \; (x,t) \in Q_{ext},  \\
u(x,t) & = & g(x,t) & & \text{for} \; (x,t) \in \Sigma,\\
u(x,0) = \partial_t u(x,t)_{|t=0} & = & 0 & & \text{for} \; x \in \Omega.
\end{array}
\end{equation}

The fundamental solution of the wave equation in $1d$ is 
$G(x,t) =  \dfrac{1}{2} \, \mathsf{H}(t-\abs{x})$,
with $\mathsf{H}$ the Heaviside step function.\footnote{While $G(x,t) = 
\dfrac{1}{2\pi}\,\dfrac{\mathsf{H}(t-\abs{x})}{\sqrt{t^2-\abs{x}^2}}$ in $2d$, 
and $G(x,t) = \dfrac{1}{4\pi} \, \dfrac{\delta(t - \abs{x})}{\abs{x}}$ in $3d$.}

Let $\mathscr{D}$ denote the \emph{double layer potential}, defined as usual as
\begin{align}
\label{eq:DL}
  (\mathscr{D}z)(x,t) &:= \int_0^t \int_\Gamma
 \partial_{n_y}  G(x-y,t-\tau) \, z(y,\tau) \, ds_y \, d\tau,
\end{align}
for $(x,t) \in Q$ and a regular enough density $z$.

We can represent the solution $u$ of \eqref{eq:IBVP} by using the double layer 
potential \cite[Ch.~1]{Say16}. Hence, we solve the following boundary integral 
equation (BIE) of the second-kind 
\begin{align}
\label{eq:BIEK}
 \left(-\frac{1}{2}\Id + \OK \right)z =  g,
\end{align}
where $\Id$ is the identity and $\OK$ the \emph{double layer operator}, which 
will be defined in Section \ref{ssec:BIO}. Similarly, we solve 
\cite[Sec.~4]{Cos94}
\begin{align}
\label{eq:BIEKex}
 \left(\frac{1}{2}\Id + \OK \right)z =  g,
\end{align}
when considering the exterior problem \eqref{eq:IBVP_exterior}. We dedicate the 
next section to introduce the mathematical framework required to properly state 
these BIEs.

%%%%%%%%%%%%%%%%%%%%%%%%%%%%%%%%%%%%%%%%%%%%%%%%%%%%%%%%%%%
\subsection{Mathematical framework for $1d$}
\label{ssec:notation1d}
Let ${\mathcal{O}} \subseteq \mathbb{R}^m, m=1,2$. We stick to the usual 
notation for the space $L^2({\mathcal{O}})$ of Lebesgue square integrable 
functions; and the Sobolev spaces $H^s({\mathcal{O}})$. Moreover, inner 
products of Hilbert spaces $X$ are denoted by standard brackets $(\cdot,\cdot
)_X$, while angular brackets $\dual{\cdot}{\cdot}_{\mathcal{O}}$ are used for 
the duality pairing induced by the extension of the inner product $(\cdot, \cdot
)_{L^2({\mathcal{O}})}$. For a Hilbert space $X$ we denote by $X'$ its dual with 
duality pairing $\dual{\cdot}{\cdot}_{\mathcal{O}}$.

We consider the spaces
\begin{eqnarray*}
  H^1_{0,}(0,T;L^2(\Omega))
  & := & \Big \{
  v \in L^2(Q) : \partial_t v \in L^2(Q), \; v(x,0)=0 \quad \mbox{for} \;
  x \in \Omega \Big \}, \\
  H^1_{,0}(0,T;L^2(\Omega))
  & := & \Big \{
  v \in L^2(Q) : \partial_t v \in L^2(Q), \; v(x,T)=0 \quad \mbox{for} \;
  x \in \Omega \Big \} .
\end{eqnarray*}
With these, we introduce
\begin{eqnarray*}
  H^{1}_{;0,}(Q)
  & := & L^2(0,T;H^1(\Omega)) \cap H^1_{0,}(0,T;L^2(\Omega)), 
  \end{eqnarray*}
which has zero initial values at $t=0$; and
 \begin{eqnarray*} 
  H^{1}_{;,0}(Q)
  & := & L^2(0,T;H^1(\Omega)) \cap H^1_{,0}(0,T;L^2(\Omega)),
\end{eqnarray*}
which prescribes zero final values at $t=T$. We remark that these are Hilbert 
spaces with norms $\norm{ u }{H^{1}_{;0,}(Q)} := \abs{u}_{H^1(Q)}$ and 
$\norm{ v }{H^{1}_{;,0}(Q)}:= \abs{u}_{H^1(Q)},$
where $\abs{u}_{H^1(Q)}:=\sqrt{\norm{ \partial_t u }{L^2(Q)}^2
         + \norm{ \nabla_x u }{L^2(Q)}^2}$.

We note that the space $L^2(\Sigma)$ boils down to
$L^2(\Sigma)=L^2(0,T)\times L^2(0,T)$ in $1d$. Furthermore, we consider the spaces
\begin{align}
    H^1_{0,}(\Sigma) 
    &:=\{ v =\begin{pmatrix}
        v_0 \\ v_L
    \end{pmatrix}: v_0, v_L \in H^1(0,T), v_0(0)=v_L(0)=0\},
\end{align}
and
\begin{align}
    H^1_{,0}(\Sigma) 
    &:=\{ v =\begin{pmatrix}
        v_0 \\ v_L
    \end{pmatrix}: v_0, v_L \in H^1(0,T), v_0(T)=v_L(T)=0\}.
\end{align}
Finally, we define $H^{1/2}_{0,}(\Sigma)$, and $H^{1/2}_{,0}(\Sigma)$ by 
interpolation as  \cite[Sect.~2.1.7]{SaS10}
\begin{align*}
  H^{1/2}_{0,}(\Sigma) :=
  [H^1_{0,}(\Sigma), L^2(\Sigma)]_{\theta=1/2}, \qquad
  H^{1/2}_{,0}(\Sigma) :=
  [H^1_{,0}(\Sigma), L^2(\Sigma)]_{\theta=1/2}.
\end{align*}
Then we have the 
following result:
\begin{lemma}[{\cite[Lemmas~3.1 and 3.2]{StU22}}]
\label{lemma:traceH10I}
  The mappings 
  \begin{align*}
   \gamma_{\Sigma}^i \, : \, H^{1}_{;0,}(Q) \to H^{1/2}_{0,}(\Sigma), \qquad
   \gamma_{\Sigma}^i \, : \, H^{1}_{;,0}(Q) \to H^{1/2}_{,0}(\Sigma),
  \end{align*}
are continuous and surjective.
\end{lemma}

%%%%%%%%%%%%%%%%%%%%%%%%%%%%%%%%%%%%%%%%%%%%%%%%%%%%%%%%%%%
\subsection{Layer potentials and boundary integral operators}
\label{ssec:BIO}

For the sake of clarity, we briefly recall the definition of the 
\emph{double layer potential} $\mathscr{D}$ from \eqref{eq:DL}: 
\begin{align*}
  (\mathscr{D}z)(x,t) &:= \int_0^t \int_\Gamma
 \partial_{n_y}  G(x-y,t-\tau) \, z(y,\tau) \, ds_y \, d\tau,
\end{align*}
for $(x,t) \in Q$ and $z$ a regular enough density.
Similarly, we also introduce the \emph{single layer potential} $\mathscr{S}$
\begin{align*}
 (\mathscr{S}w)(x,t) &:= \int_0^t \int_\Gamma G(x-y,t-\tau)
 \, w(y,\tau) \, ds_y \, d\tau,
\end{align*}
for $(x,t) \in Q$ and $w$ a regular enough density.

Let $\gamma_N^i$ be the interior Neumann trace, which is defined 
variationally as in \cite{Cos88}. Without loss of generality, let us consider 
$\Omega^c := B_R\setminus\overline{\Omega}$ and $Q^c:=\Omega^c \times (0,T)$,
with $B_R := \lbrace \Vx \in \R^n \, : \, \abs{\Vx}<R\rbrace$ a sufficiently 
large ball containing $\Gamma$. This allows us to define the exterior traces 
$\gamma_\Sigma^e$ and $\gamma_N^e$ in analogy to $\gamma_\Sigma^i$ and 
$\gamma_N^i$, but using $Q^c$ instead of $Q$. With this, we can define the BIOs 
in the classical way as \cite[Sec.~4]{Cos94}
 \begin{align*}
 \text{\it Weakly singular operator}: 
 && \OV w &:= \gamma_{\Sigma}^i \mathscr{S} w 
 = \gamma_{\Sigma}^e \mathscr{S} w; \\
\text{\it Double layer operator}: 
&& \OK z &:= \frac{1}{2}\left( \gamma_{\Sigma}^i \mathscr{D} z
 + \gamma_{\Sigma}^e \mathscr{D} z \right);\\
\text{\it Adjoint double layer operator}:
&&  \OK' w &:= \frac{1}{2}\left( \gamma_N^i \mathscr{S} w
 + \gamma_N^e \mathscr{S} w \right); \\
\text{\it Hypersingular operator}: 
&&   \OW z &:= -\gamma_N^i \mathscr{D} z
  = -\gamma_N^e \mathscr{D} z \, .
\end{align*}
As already pointed out, the double layer operator $\OK$ is the main character 
in this paper. Therefore, we discuss its properties in the 
next subsection\footnote{and refer the interested reader to \cite[Sec.~5]{StU22} 
for further details about the other operators}.

%%%%%%%%%%%%%%%%%%%%%%%%%%%%%%%%%%%%%%%%%%%%%%%%%%%%%%%%%%%%%%%%%%%%%%%%%%%%%%%
\section{Main Results}
\label{sec:2ndKF}

%%%%%%%%%%%%%%%%%%%%%%%%%%%%%%%%%%%%%%%%%%%%%%%%%%%%%%%%%%%
\subsection{Mapping induced by the double layer operator $\OK$}
\label{ssec:mapK}

We now focus on studying the effect of applying $\OK$ to a function 
in our trace space, as this will be needed to prove ellipticity later on. 

Let $v\in H^{s}_{0,}(\Sigma)$ with $s\in[0,1]$, in $1d$ we can express $\OK v$ 
as:
\begin{equation}
\label{eq:defDL}
\OK v (x,t):=  \frac{1}{2}\int_\Sigma \partial_{\Vn_y}\mathsf{H}(x-y,t-\tau)v(y,
\tau) ds_y d\tau, \quad (x,t) \in \Sigma,
\end{equation}
with $\partial_{\Vn_y}$ the normal derivative with respect to $y$.

Let $\Vn(x)$ denote the outward normal vector with respect to $\Gamma$ at $x$. 
Since $\Gamma:=\{0,L\}$, we have in this case $\Vn(0)=-1$ and $\Vn(L)=1$. Then, 
one can verify that 
\begin{equation}
\label{eq:normalderivativeH}
\partial_{\Vn_y}\mathsf{H}(x-y,t-\tau)= \Vn(y)\sgn{x-y}\delta(t-\tau-\abs{x-y}),
\end{equation}
where $\sgn{\cdot}$ denotes the signum function and $\delta(\cdot)$ the Dirac 
delta distribution. Combining \eqref{eq:defDL} and \eqref{eq:normalderivativeH}, 
we get
\begin{align}
\label{eq:expressionDL}
    \OK v(x,t) = &-\frac{1}{2}\int_0^T \sgn{x}\delta(t-\tau -\abs{x})v(0,\tau) 
    \;d\tau\nonumber \\
    &+ \frac{1}{2}\int_0^T \sgn{x-L}\delta(t-\tau-\abs{x-L})v(L,\tau) \; d\tau.
\end{align}

Since $\sgn{0}=0$, we can further reduce \eqref{eq:expressionDL} to
\begin{equation}
\label{eq:finalformK}
    \OK v(x,t) = \begin{cases} -\frac{1}{2}v(L,t-L) & x=0,\\
    -\frac{1}{2}v(0,t-L)& x=L.           
    \end{cases}, \quad \text{ for }v\in H^{s}_{0,}(\Sigma),\, s\in[0,1].
\end{equation}
\begin{remark}\label{rem:mapping_K}
    We gather from \eqref{eq:finalformK} that the application of $\OK$ to a 
    function $v$ simply translates $v$ in time and flips its position in space 
    by interchanging the evaluation on $x=0$ by $x=L$, and vice-versa. We will 
    exploit this when we prove ellipticity.
\end{remark}

%%%%%%%%%%%%%%%%%%%%%%%%%%%%%%%%%%%%%%%%%%%%%%%%%%%%%%%%%%%
\subsection{Continuity and ellipticity in $L^2(\Sigma)$}
\label{ssec:CEL2}

Using \eqref{eq:finalformK}, we have that:
\begin{equation*}
    \left(\pm\frac{1}{2}\Id+\OK\right) w(x,t) := \begin{cases} 
    \pm\frac{1}{2} w(0,t) - \frac{1}{2}w(L,t-L) &x=0\\
    \pm\frac{1}{2} w(L,t) - \frac{1}{2}w(0,t-L) & x=L.
    \end{cases}
\end{equation*}

Since we focus on second-kind formulations \eqref{eq:BIEK} and 
\eqref{eq:BIEKex}, we need to understand where these BIEs make sense. 
Let us begin with the scenario when the Dirichlet data $g$ is in $L^2(\Sigma)$.
Since we are interested in the Dirichlet problems \eqref{eq:IBVP}-
\eqref{eq:IBVP_exterior}, let us consider $w\in L^2(\Sigma)$ such that $w(\cdot,
t) = 0$ for all $t\leq 0$. Then, for $t\leq L$ we get
\begin{equation*}
    \norm{\left(\pm\frac{1}{2}\Id+\OK\right) w}{L^2(\Sigma)} \leq  \norm{\pm
    \frac{1}{2} w}{L^2(\Sigma)}\leq \frac{1}{2}\norm{w}{L^2(\Sigma)},
\end{equation*}
while for $t> L$
\begin{equation*}
 \norm{\left(\pm\frac{1}{2}\Id+\OK\right) w}{L^2(\Sigma)} \leq  \norm{\pm\frac{
 1}{2} w}{L^2(\Sigma)} + \norm{-\frac{1}{2}w(\cdot,t-L)}{L^2(\Sigma)} \leq 
 \norm{w}{L^2(\Sigma)}.
\end{equation*}

Hence, we have continuity of the operators $\left(\pm\frac{1}{2}\Id+\OK\right): 
L^2(\Sigma)\to L^2(\Sigma)$. 

To show ellipticity in $L^2(\Sigma)$, we will follow the proof of 
Theorem 2.1 in \cite{SUZ21}. For this, we introduce 
\begin{equation}\label{eq:time_slices}
n := \min \Big \{ m \in {\mathbb{N}} : T \leq m L \Big \},
\end{equation}
which is the number of time slices $T_j := ((j-1)L,j L)$ for $j =1,\ldots,n$
in the case $ T=nL$. In the case $ T < nL$, we define the last time slice
as $T_n:=((n-1)L,T)$, while all the others remain unchanged.

\vspace{1em}
\noindent\fbox{\parbox{0.975\textwidth}{
\begin{theorem}
\label{thm:L2ellip}
  For all $w \in L^2(\Sigma)$, there holds the ellipticity estimate
\begin{equation}\label{eq:L2Ell}
\abs{\inner{\left(\pm\frac{1}{2}\Id+\OK\right) w }{w}_{L^2(\Sigma)}} \, \geq \,
     c_s(n) \, \| w \|_{L^2(\Sigma)}^2,
\end{equation}
where the number $n \in {\mathbb{N}}$ of time slices is defined in 
\eqref{eq:time_slices}, and 
\begin{equation}
 c_s(n):= \sin^2 \left(\frac{\pi}{2(n+1)}\right) \quad n\geq 1.
\end{equation}
\end{theorem}
}}

\begin{proof}
When $t\leq L$, we have that
\begin{equation}
        \abs{\inner{\left(\pm\frac{1}{2}\Id+\OK\right) w}{w}_{L^2(\Sigma)}} = 
        \abs{\pm\frac{1}{2}\inner{ w}{w}_{L^2(\Sigma)}} 
        = \frac{1}{2}\norm{w}{L^2(\Sigma)}^2,
    \end{equation}
so the operator is $L^2(\Sigma)$- elliptic with constant $\dfrac{1}{2}$.

Let $T_j$ denote the $j-th$ time-slice, $j=1,\dots,n$. Clearly, $t> L$ implies 
that we assume $n>1$. Given $w \in L^2(\Sigma)$:
\begin{align*}
\abs{2\inner{\left(\pm\frac{1}{2}\Id+\OK\right) w}{w}_{L^2(\Sigma)}} &
= \abs{2\inner{\left(\frac{1}{2}\Id\pm\OK\right) w}{w}_{L^2(\Sigma)}}\\
&=\left|\norm{w(0,\cdot)}{L^2(\Sigma)}^2 +\norm{w(L,\cdot)}{L^2(\Sigma)}^2 \right.\\
&\hphantom{=|} \mp \int_{\Sigma} w(L,t)w(0,t-L)\; dt \\
&\hphantom{=|} \left.\mp  \int_{\Sigma} w(0,t)w(L,t-L)\; dt\right|.
\end{align*}
Now, we notice this is the same we had in the proof of Theorem 2.1 in 
\cite[page 4]{SUZ21}. We therefore use the same computations derived there to get 
the following: 
\begin{align*}    
&\abs{2\inner{\left(\frac{1}{2}\Id\pm\OK\right) w}{w}_{L^2(\Sigma)}}\\ 
&=\left|\sum_{j=1}^n\left[ \norm{w(0,\cdot)}{L^2(T_j)}^2 
+\norm{w(L,\cdot)}{L^2(T_j)}^2 \right.\right.\\
& \hphantom{=|}\mp \left.\left. \int_{T_j} w(L,t)w(0,t-L)\; dt 
\mp\int_{T_j} w(0,t)w(L,t-L)\; dt \right]\right|\\    
% \end{align*}
% \begin{align*} 
& \geq \left| \sum_{j=1}^n \left [  \norm{w(0,\cdot)}{L^2(T_j)}^2 
+\norm{w(L,\cdot)}{L^2(T_j)}^2 \right]\right. \\  
& \hphantom{=|}- \left.\sum_{j=2}^n \left[ \norm{w(0,\cdot)}{L^2(T_j)}
\norm{w(L,\cdot)}{L^2(T_{j-1})} +\norm{w(L,\cdot)}{L^2(T_j)}\norm{w(0,\cdot)}
{L^2(T_{j-1})} \right]\right| \\
& \geq \sum_{j=1}^n \left [  \norm{w(0,\cdot)}{L^2(T_j)}^2 +\norm{w(L,\cdot)}
{L^2(T_j)}^2 \right]\\
& \hphantom{=}- \sum_{j=2}^n \left[ \norm{w(0,\cdot)}{L^2(T_j)}\norm{w(L,\cdot)}
{L^2(T_{j-1})} +\norm{w(L,\cdot)}{L^2(T_j)}\norm{w(0,\cdot)}{L^2(T_{j-1})} \right]\\
&\geq 2\sin^2 \frac{\pi}{2(n+1)}\norm{w}{L^2(\Sigma)}^2,
\end{align*}
which shows the ellipticity of $\left(\pm\frac{1}{2}\Id+\OK\right)$ in $L^2(\Sigma)$ 
with ellipticity constant $c_s(n):= \sin^2\dfrac{\pi}{2(n+1)}$.\\
\end{proof}

By Lax-Milgram, the continuity and ellipticity in $L^2$ implies that 
\eqref{eq:BIEK} and \eqref{eq:BIEKex} have unique solutions. Moreover, by choosing 
a conforming discretisation, we will also get unique solvability of the arising 
linear system.

%%%%%%%%%%%%%%%%%%%%%%%%%%%%%%%%%%%%%%%%%%%%%%%%%%%%%%%%%%%
\subsection{Existence and uniqueness in $H^1_{0,}(\Sigma)$ and $H^{1/2}_{0,}
(\Sigma)$}
\label{ssec:CEH1}

We first turn our attention to the scenario when the Dirichlet data $g$ is in 
$H^1_{0,}(\Sigma)$, although ultimately, what we truly aim for is the case $g\in H^{1/2}_{0,}(\Sigma)$.
With this purpose in mind, we obtain the following result:

\vspace{1em}
\noindent\fbox{\parbox{0.975\textwidth}{
\begin{theorem}
\label{thm:K1dH1}
    In 1d, the operators  $\pm\frac{1}{2} \Id + \OK : H^{1}_{0,}(\Sigma) \to 
H^{1}_{0,}(\Sigma)$ are continuous and isomorphisms. 
\end{theorem}
}}
\begin{proof}
$\hphantom{.}$\linebreak
\noindent\underline{\emph{Step 1:}} 
We begin by recalling from 
\cite[Thm.~20]{Gua10} that $\OW$ satisfies the ellipticity estimate
\begin{equation*}
\label{eq:W_Guardasoni}
\dual{\OW \phi}{\partial_t \phi}_\Sigma \geq c(T) \norm{\partial_t \phi}{L^2(
\Sigma)}^2, \quad \forall \phi \in H^1_{0,}(\Sigma),
\end{equation*}
for some constant $c(T)>0$ depending on the terminal time $T$. The fact that 
$\partial_t:H^1_{0,}(\Sigma) \to L^2(\Sigma)$ is an isomorphism 
\cite[p.~6]{SUZ21} implies the existence of its inverse, which we denote by 
$\partial_t^{-1}$.

By taking $v = \partial_t^{-1} \phi$, for $\phi\in H^1_{
0,}(\Sigma)$ we obtain
\begin{equation*}
\label{eq:W_isomorphism}
    \dual{\OW \partial_t^{-1}v}{v}_\Sigma \geq c(T) \norm{ v}{L^2(\Sigma)}^2, 
    \quad \forall v\in L^2(\Sigma),
\end{equation*}
i.e., $\OW \partial_t^{-1}: L^2(\Sigma)\to L^2(\Sigma)$
defines an isomorphism. Hence $\OW:  H^1_{0,}(\Sigma)\to L^2(\Sigma)$
is also an isomorphism. 

\noindent\underline{\emph{Step 2:}} 
Finally, using that in $1d$ $\OV: L^2(\Sigma)\to H^{1}_{0,}(\Sigma)$ is 
continuous and an isomorphism, as shown in \cite[Sec.~2]{SUZ21}, we conclude 
the desired result from the Calder\'on identity \cite[p.~18]{StU22} 
\begin{equation}
\label{eq:CalderonId}
    \OV \OW = \left (\frac{1}{2}\Id - \OK\right) 
    \left(\frac{1}{2}\Id + \OK\right).
\end{equation}
\end{proof}

Similarly, we can prove:\\
\vspace{1em}
\noindent\fbox{\parbox{0.975\textwidth}{
\begin{theorem}
\label{thm:K1dHhalf}
In 1d, the operators  $\pm\frac{1}{2} \Id + \OK : H^{1/2}_{0,}(\Sigma) \to 
H^{1/2}_{0,}(\Sigma)$ are continuous and isomorphisms. 
\end{theorem}
}}
\vspace{-2em}
\begin{proof}
$\hphantom{.}$\linebreak
\noindent\underline{\emph{Step 1:}} By step 1 in the proof of 
Theorem~\ref{thm:K1dH1}, we have that $\OW:  H^1_{0,}(\Sigma)\to L^2(\Sigma)$
is an isomorphism. 

\noindent\underline{\emph{Step 2:}} By duality, we get that $\OW:  L^2(\Sigma)
\to [H^{1}_{,0}(\Sigma)]^\prime$ is an isomorphism. Hence, by 
interpolation, $\OW:  H^{1/2}_{0,}(\Sigma)\to [H^{1/2}_{,0}(\Sigma)]^\prime$
is also an isomorphism.

\noindent\underline{\emph{Step 3:}} 
Finally, we proceed as in step 1 in the proof of Theorem~\ref{thm:K1dH1}, but 
now using that in $1d$ $\OV: [H^{1/2}_{,0}(\Sigma)]^\prime\to H^{1/2}_{0,}
(\Sigma)$ is continuous and an isomorphism, as shown in \cite[Sec.~2]{SUZ21}.
\end{proof}

Theorems~\ref{thm:K1dH1} and \ref{thm:K1dHhalf} imply that the operators satisfy 
their corresponding inf-sup conditions, which would be enough to have existence 
and uniqueness for the related boundary integral equations of the second kind. 
Yet, they will not suffice to guarantee that we satisfy the required discrete 
inf-sup conditions. This will be addressed in the next Section.

%%%%%%%%%%%%%%%%%%%%%%%%%%%%%%%%%%%%%%%%%%%%%%%%%%%%%%%%%%%%%%%%%%%%%%%%%%%%%%%
\section{Discretisation}
\label{sec:dis}

Let $\Sigma_h$ be a mesh of the lateral boundary $\Sigma$, consisting of 
non-overlapping intervals and with maximal meshwidth $h$. Furthermore, we 
assume the spatial part of the support of each element is restricted to either
the left ($x=0$) or the right ($x=L$) boundary. In other words, we can write
$\Sigma_h = \Sigma_h^{0}\cup\Sigma_h^{L}$, with $\Sigma_h^{0}$ a partition of
the left boundary $\lbrace x=0 \rbrace \times (0,T)$, and  $\Sigma_h^{L}$ a 
partition of the right boundary $\lbrace x=L \rbrace \times (0,T)$.

We consider the following boundary element spaces 
\begin{itemize}
\item $\calS^0(\Sigma_h)$ that is spanned by space-time piecewise constant 
basis functions on $\Sigma_h$;
\item $\calS^1_{{0,}}(\Sigma_h)$ that is spanned by space-time piecewise 
linear basis functions on $\Sigma_h$ \emph{with zero-initial conditions}.
\end{itemize}

As stated in Theorem \ref{thm:L2ellip}, the operator $-\frac{1}{2}\Id +\OK$ is 
continuous and elliptic in $L^2(\Sigma)$. Hence, for given Dirichlet data $g_D 
\in L^2(\Sigma)$, we the following conforming Galerkin formulation will be 
uniquely solvable:\\ 
\vspace{1em}
\noindent\fbox{\parbox{0.975\textwidth}{%
Find $z_h^0\in \calS^0(\Sigma_h)$ such that
\begin{equation}\label{eq:BIEKdis1d_pwc}
    \inner{\left(-\frac{1}{2}\Id +\OK\right)z_h^0}{q_h}_{L^2(\Sigma)} = 
    \inner{g_D}{q_h}_{L^2(\Sigma)},
\end{equation}
for all $q_h \in \calS^0(\Sigma_h)\subset L^2(\Sigma)$.
}}

Additionally, inspired by Theorem~\ref{thm:K1dHhalf}, we pursue the following 
Galerkin discretisation of \eqref{eq:BIEK}:\\ 
\vspace{1em}
\noindent\fbox{\parbox{0.975\textwidth}{%
Find $z_h^1\in \calS^1_{{0,}}(\Sigma_h)\subset H^{1/2}_{0,}(\Sigma)$ such that
\begin{align}
\label{eq:BIEKdis1d_pwl}
\dual{\left(-\frac{1}{2}\Id+\OK\right)z_h^1}{v_h}_{\Sigma} = \dual{g}{v_h}_{
\Sigma},
\end{align}
for all $v_h \in \calS^1_{{0,}}(\Sigma_h)\subset[H^{1/2}_{0,}(\Sigma)]^\prime$.
}}
% \vspace{1em}

Note that our results from Section~\ref{sec:2ndKF} also allow us to consider 
\eqref{eq:BIEKdis1d_pwl} in $H^1_{0,}(\Sigma)$, but omit it for conciseness.

We obviously proceed analogously for \eqref{eq:BIEKex}, but we will refrain 
from writing the formula for the sake of brevity. In order to have stability, 
we need to introduce one requirement on the family of meshes that we can use:

\begin{definition}
\label{def:adm}
Let $X_h$ be a boundary element space. A mesh $\Sigma_h$ of $\Sigma$ is 
$X_h$-admissible if the $L^2$-orthogonal projection onto $X_h$ is $H^1$-stable.
\end{definition}

\begin{remark}
The $H^1$-stability in Definition~\ref{def:adm} is a well studied requirement 
in Galerkin discretisations. It always holds for globally quasi-uniform meshes
and piecewise polynomial finite element spaces \cite{BrX91,Ste03} (and hence 
in particular for $\calS^1(\Sigma_h)$). For locally refined meshes, one needs 
to (i) impose a-priori bounds on the gradings of the mesh, as in 
\cite{BPS02,Car02,CrT87,ErJ95,Ste02}; 
(ii) impose a fixed refinement strategy \cite{BaY14,Car04,KPP13}; or 
(iii) a combination of both \cite{KPP14,GSU21}.
\end{remark}

\begin{figure}[h]
\centering
\begin{tikzpicture}
% [every rectangle node/.style={draw},
% every circle node/.style={draw,double}]

\shade[left color=white,right color=lightgray] (0,0) rectangle (5,5) node[pos=.5] {$T_1$};
\shade[left color=lightgray,right color=white] (0,5) rectangle (5,10) node[pos=.5] {$T_2$};
\def\lastx{-0.05}
\draw[black] (0,0)--(0,10);
\draw[black] (5,0)--(5,10);
\foreach \x[remember =\x as \lastx] in {0,0.4, 1, 2, 4 ,5}{
    \draw (0,\x) node[rectangle, fill =orange, scale = 0.4, draw=black, double] {};
    \draw (5,\x+5) node[rectangle, fill =orange, scale = 0.4, draw= black,double] {};
%       \draw[ultra thick, lightgray] (5,\lastx+5.05)--(5,5+\x);
%       \draw[ultra thick, lightgray] (0,\lastx+0.05)--(0,0+\x);
    }
\def\lastx{-0.05}
\foreach \x[remember =\x as \lastx] in {0,1,2.5,2.8,4.3,5}{
    \draw (0,\x+5) node[circle, fill =blue, scale = 0.4, draw= black, double] {};
    \draw (5,\x) node[circle, fill =blue, scale = 0.4, draw= black, double] {};
%        \draw[ultra thick, black] (0,\lastx+5.05)--(0,5+\x);
%        \draw[ultra thick, black] (5,\lastx+0.05)--(5,0+\x);
    }
\draw[dotted] (-1.5,0)  -- (0,0) node[pos=-0.5] {$t=0$};
\draw[dotted] (-1.5,5)  -- (0,5) node[pos=-0.5] {$t=L$};
\draw[dotted] (-1.5,10)  -- (0,10) node[pos=-0.5] {$t=2L$};
\draw[dotted] (0,-1)  -- (0,0) node[pos=-0.3] {$\Sigma_h^0$};
\draw[dotted] (5,-1)  -- (5,0) node[pos=-0.3] {$\Sigma_h^L$};

\end{tikzpicture}%
\caption{Example of non-uniform mesh satisfying 
Assumption~\ref{Assumption:mesh}. On each time-slice $T_j$, the degrees of 
freedom (DoFs) of each boundary agree with the DoFs on the opposite boundary 
shifted in time by $L$.}
\label{fig:Example_mesh}
\end{figure}
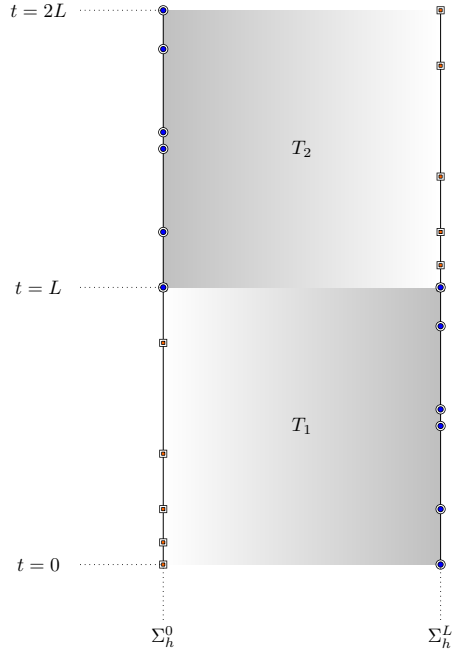

In order to ensure discrete inf-sup stability, we impose some additional 
constraints on the $\calS^1_{{0,}}(\Sigma_h)$-admissible mesh. We assume the 
following mesh condition holds:
\begin{assumption}\label{Assumption:mesh}
   Let us assume $T = nL$ for some $n\in \N$ and consider the time-slices 
   $T_j:=((j-1)L,jL)$ for $j=1,\dots,n$. Let $\Sigma_h$ be mesh of $\Sigma$. 
   We further assume that the mesh on the boundary restricted to a time-slice 
   corresponds to the mesh on the opposite boundary restricted to the 
   subsequent time-slice. In other words, we require:
   \begin{align}
       \Sigma^0_{h|T_j}= \Sigma^L_{h|T_{j+1}}, \text{ and }\: \Sigma^L_{h|T_j}=
       \Sigma^0_{h|T_{j+1}}, \quad \forall j=1,\dots,n-1.
   \end{align}
\end{assumption}

Clearly, all uniform meshes satisfy this assumption. An example of a 
non-uniform mesh satisfying assumption \ref{Assumption:mesh} is given in 
Figure~\ref{fig:Example_mesh}.

% \newpage
Then, we can show the following result:

\noindent\fbox{\parbox{0.975\textwidth}{
\begin{theorem}
\label{thm:DinfsupK1d}

Let $\Sigma_h$ be an $\calS^1_{{0,}}(\Sigma_h)$-admissible
mesh of $\Sigma$ such that Assumption \ref{Assumption:mesh} holds.
Then, the following discrete inf-sup conditions are satisfied in 1d
  \begin{equation}
  \label{eq:infsupK1d}
    c_1^{\OK} \, \norm{ w_h }{H^{1/2}_{0,}(\Sigma)} \leq 
    \sup\limits_{0 \neq \mu_h \in \calS^1_{{0,}}(\Sigma_h)}
    \frac{\abs{\dual{\left(\pm\frac{1}{2}\Id+\OK\right) w_h}{\mu_h}_\Sigma}}
    {\norm{\mu_h}{[H^{1/2}_{0,}(\Sigma)]^\prime}},
  \end{equation}
  for all $w_h \in \calS^1_{{0,}}(\Sigma_h)$.
\end{theorem} 
}}
\vspace{1em}

We dedicate the remaining of this subsection to prepare the required 
ingredients to prove this Theorem.

%%%%%%%%%%%%%%%%%%%%%%%%%%%%%%%%%%%%%%%%%%%%%%%%%%%%%%%%%%%
\subsection{Proof of Theorem~\ref{thm:DinfsupK1d}}
\label{ssec:auxLemmas}

We begin by introducing the $L^2$ projection $Q^0_h: L^2(\Sigma)\to 
\calS^1_{0,}(\Sigma_h)$ as 
\begin{equation}
\label{eq:L2project}
    \dual{Q^0_h u}{v_h}_{\Sigma} = \dual{u}{v_h}_{\Sigma}, \quad 
    \forall v_h \in \calS^1_{0,}(\Sigma_h).
\end{equation}
Let $Q^{0*}_h: L^2(\Sigma) \to \calS^1_{0,}(\Sigma_h) \subset L^2(\Sigma)$ be 
the unique (Hilbert space-)adjoint operator. We note that in this case 
$Q^{0*}_h=Q^0_h$.

In order to proceed, we need to show the following auxiliary result
\begin{lemma}
\label{lem:Qh_stab_dualspace}

For $\calS^1_{{0,}}(\Sigma_h)$-admissible meshes
\begin{equation}
\label{eq:stabdual_recalled}
 \norm{Q^0_h \nu}{[H^{1/2}_{0,}(\Sigma)]^\prime} \leq c_s \norm{\nu}{[H^{1/2}_{
 0,}(\Sigma)]^\prime}, \qquad \forall \nu \in [H^{1/2}_{0,}(\Sigma)]^\prime,
\end{equation}
is satisfied with a constant $c_s>0$.
\end{lemma}
\begin{proof}
We have that
\begin{equation}
 \norm{Q^0_h w}{L^2(\Sigma)} \leq \norm{w}{L^2(\Sigma)}, \qquad \forall w \in 
 L^2(\Sigma),
\end{equation}
and, from the definition of $\calS^1_{{0,}}(\Sigma_h)$-admissibility, that 
\begin{equation}
 \norm{Q^0_h w}{H^1_{0,}(\Sigma)} \leq c_{s,1} \norm{w}{H^1_{0,}(\Sigma)}, 
 \qquad \forall w \in H^1_{0,}(\Sigma).
\end{equation}

Hence, by interpolation of bounded linear operators, we get that
\begin{equation}
 \norm{Q^0_h w}{H^{1/2}_{0,}(\Sigma)} \leq \sqrt{c_{s,1}} \norm{w}
 {H^{1/2}_{0,}(\Sigma)}, \qquad \forall w \in H^{1/2}_{0,}(\Sigma).
\end{equation}
Moreover, by \cite[Theorem~2.1]{Ste03}, this is equivalent to 
\begin{equation}
\label{eq:L2_proj_dualH1bound}
 \norm{Q^{0*}_h \nu}{[H^{1/2}_{0,}(\Sigma)]^\prime} \leq \sqrt{c_{s,1}} 
 \norm{\nu}{[H^{1/2}_{0,}(\Sigma)]^\prime}, \qquad \forall \nu \in 
 [H^{1/2}_{0,}(\Sigma)]^\prime,
\end{equation}
Finally, the result follows from $Q^{0*}_h=Q^0_h$ in our case. 
\end{proof}

Now we have all the pieces to show our main theorem.\\
\textit{Proof of Theorem~\ref{thm:DinfsupK1d}.}
Without loss of generality, let us prove our result for $\left(-\frac{1}{2}\Id
+\OK\right)$ and note that $\left(\frac{1}{2}\Id+\OK\right)$ follows 
analogously. From \cite[Thm.~5.6]{StU22} we know that $\left(\frac{1}{2}\Id
-\OK\right)$ satisfies the continuous inf-sup condition, we henceforth can 
pick a $\mu_w\in [H^{1/2}_{0,}(\Sigma)]^\prime$ such that
\begin{equation}
\label{eq:defn_mu_z}
\frac{\abs{  \dual{\left(-\frac{1}{2}\Id +\OK\right)w}{\mu_w}_\Sigma}}
{\norm{\mu_w}{[H^{1/2}_{0,}(\Sigma)]^\prime}} \geq \alpha_{\OK} \norm{w}{H^{1/2}_{0,}
(\Sigma)}, \quad \forall w\in H^{1/2}_{0,}(\Sigma),
\end{equation}
for some $0<\alpha_{\OK} <\infty$.

Let $w_h \in \calS_{0,}^1(\Sigma_h)$. Note that $w_h\in H^{1/2}_{0,}(\Sigma)$, 
hence \eqref{eq:defn_mu_z} holds with a $\mu_{w_h}\in [H^{1/2}_{0,}(\Sigma)
]^\prime$. In general, given the mapping properties of $\OK$ (see 
\eqref{eq:finalformK}), we \textit{cannot} assume $\OK w_h \in \calS_{0,}^1(
\Sigma_h)$. Nevertheless, as $\OK$ induces a translation in time of magnitude 
$L$, the function $\OK w_h$ remains piecewise linear, and 
Assumption~\ref{Assumption:mesh} \emph{guarantees} that $(-\Id +\OK)w_h \in 
\calS_{0,}^1(\Sigma_h)$. 

Considering this, we derive
\begin{align*}
\sup_{0\neq \mu_h\in \calS_{0,}^1(\Sigma_h)} 
\frac{\abs{\dual{\left(-\frac{1}{2}\Id +\OK\right)w_h}{\mu_h}_\Sigma}}
{\norm{\mu_h}{[H^{1/2}_{0,}(\Sigma)]^\prime}} 
&\geq \frac{\abs{\dual{\left(-\frac{1}{2}\Id +\OK\right)w_h}
{Q^0_h\mu_{w_h}}_\Sigma}}{\norm{Q^0_h\mu_{w_h}}{[H^{1/2}_{0,}(\Sigma)]^\prime}} 
\nonumber\\
&\stackrel{\eqref{eq:L2_proj_dualH1bound}}{\geq} 
\frac{\abs{\dual{\left(-\frac{1}{2}\Id +\OK\right)w_h}{Q^0_h\mu_{w_h}}_\Sigma}}
{c_s\norm{\mu_{w_h}}{[H^{1/2}_{0,}(\Sigma)]^\prime}} \\
&\stackrel{\eqref{eq:L2project}}{=}  
\frac{\abs{\dual{\left(-\frac{1}{2}\Id +\OK\right)w_h}{\mu_{w_h}}_\Sigma}}
{c_s\norm{\mu_{w_h}}{[H^{1/2}_{0,}(\Sigma)]^\prime}}\nonumber\\
&\stackrel{\text{Thm.}~\ref{thm:K1dHhalf}}{\geq}\frac{\alpha_{\OK}}{c_s} 
\norm{w_h}{H^{1/2}_{0,}(\Sigma)},
\end{align*}
for all $w_h\in\calS_{0,}^1(\Sigma_h)$, which proves the theorem with 
$c_1^{\OK}=\dfrac{\alpha_{\OK}}{c_s} $. \\\qed

%%%%%%%%%%%%%%%%%%%%%%%%%%%%%%%%%%%%%%%%%%%%%%%%%%%%%%%%%%%%%%%%%%%%%%%%%%%%%%%
\section{Error Estimates}
\label{sec:EE}

Based on standard arguments \cite{SaS10,Ste08}, we derive a priori error 
estimates for the Galerkin approximations \eqref{eq:BIEKdis1d_pwc} and 
\eqref{eq:BIEKdis1d_pwl}. 

As a consequence of %Corollary~\ref{thm:DinfsupK1d_half} 
Theorem~\ref{thm:DinfsupK1d} and \cite[Thm.~4.2.1]{SaS10}, 
we have that:
\begin{proposition}
\label{prop:convergence_Galerkin}
Let $z\in H_{0,}^{1/2}(\Sigma)$ be the exact solution to \eqref{eq:BIEKex} and 
$z_h^1 \in \calS^1_{0,}(\Sigma_h)$ the discrete solution of 
\eqref{eq:BIEKdis1d_pwl}. The following a priori error estimate holds
\begin{equation}
\label{eq:apriori_error_estimate}
\norm{z-z_h^1}{H_{0,}^{1/2}(\Sigma)} \leq \left(1+\frac{\norm{\pm \frac{1}{2}\Id 
+\OK}{ }}{{c}_1^{\OK}}\right)\inf_{v_h\in \calS_{0,}^1(\Sigma_h)}\norm{z-v_h}{
H_{0,}^{1/2}(\Sigma)},
\end{equation}
where ${c}_1^{\OK}$ is the inf-sup constant from \eqref{eq:infsupK1d}.
\end{proposition}

Combining Proposition~\ref{prop:convergence_Galerkin} with the approximation 
property of the space of piecewise linear functions \cite[Thm.~10.9]{Ste08},
yields the a priori error estimate:
\begin{corollary}
\label{cor:error_convergence_hhalf}
Let $z\in H_{0,}^{s}(\Sigma)$ for $s\in [1/2,2]$ denote the exact solution to 
\eqref{eq:BIEK}. Then for the Galerkin solution $z^1_h \in \calS^1_{0,}(\Sigma_h)$ 
we have the a priori error estimate
\begin{equation}
\label{eq:error_convergence_hhalf}
    \norm{z-z^1_h}{H_{0,}^{1/2}(\Sigma)} \leq \tilde{c} \,h^{s-1/2}\norm{z}{
    H^s(\Sigma)},
\end{equation}
for a positive constant $\tilde{c}$ independent of $h$.
\end{corollary}

Similarly, by Theorem~\ref{thm:L2ellip}, we can use the standard tools for elliptic 
problems and the approximation property of the space of piecewise constants 
\cite[Thm.~10.4]{Ste08} to get 
\begin{corollary}
\label{cor:error_convergence_L2}
Let $z\in H_{0,}^{s}(\Sigma)$ for $s\in [0,1]$ denote the exact solution to 
\eqref{eq:BIEK}. Then for the Galerkin solution $z^0_h \in \calS^0(\Sigma_h)$ 
we have the a priori error estimate
\begin{equation}
\label{eq:error_convergence_hhalf}
    \norm{z-z^0_h}{L^2(\Sigma)} \leq \tilde{c}_0 \,h^{s}\norm{z}{
    H^s(\Sigma)},
\end{equation}
for a positive constant $\tilde{c}_0$ independent of $h$.
\end{corollary}

%%%%%%%%%%%%%%%%%%%%%%%%%%%%%%%%%%%%%%%%%%%%%%%%%%%%%%%%%%%%%%%%%%%%%%%%%%%%%%%
\section{Numerical Experiments}
\label{sec:numexp}

Before we present our numerical results, we dedicate the next two subsections 
to briefly describe a couple of tools that we will use.

%%%%%%%%%%%%%%%%%%%%%%%%%%%%%%%%%%%%%%%%%%%%%%%%%%%%%%%%%%%
\subsection{Norm approximation}
The error is 
measured in the $H^{1/2}_{0,}(\Sigma)$-norm, which can be represented using 
the modified Hilbert transform $\mathcal{H}_T$ \cite[Section~2.4]{StZ20}: 
\begin{equation*}
\dual{\partial_t u}{\mathcal{H}_T u}_\Sigma = \norm{u}{H^{1/2}_{0,}(\Sigma)}^2, 
\quad \text{for all }u \in  H^{1/2}_{0,}(\Sigma).
\end{equation*}
Computation of the $H^{1/2}_{0,}(\Sigma)$ requires a representation of the 
functions $\calH_T z$, where $z$ is the exact solution of \eqref{eq:BIEK}. As 
these terms are often challenging to retrieve, we approximate the $H^{1/2}_{0,}
(\Sigma)$-norm by interpolating $z$. Since the exact solutions we will consider 
for our experiments are globally continuous, we are allowed to define the 
interpolation of $z$ by using the (standard) nodal interpolation operator $I_h: 
C^0(\Sigma) \to \calS_{0,}^1(\Sigma_h)$.

It is worth pointing out that the expected convergence rate of the Galerkin 
approximation is not affected by the interpolation of the exact solution, and 
therefore this will not affect our numerical results.

%%%%%%%%%%%%%%%%%%%%%%%%%%%%%%%%%%%%%%%%%%%%%%%%%%%%%%%%%%%
\subsection{Exact representation of the unknown density}
\label{ssec:ExRep}

In general, solutions to the indirect methods do not yield densities which can 
be interpreted physically. This limits the possibility to find an exact 
representation of the density, solving \eqref{eq:BIEK}. However, if the 
following assumption holds on the Dirichlet data
\begin{equation}
\label{eq:assumption_Dirichlet_data}
    g(0,t-NL) = g(L,t-(N-1)L), \quad N\in \mathbb{N},\; t\geq 0,
\end{equation}
an exact solution is available. In order to see this, let us first restate the 
considered BIE, assuming a positive addition of the double layer operator
\begin{equation*}
    \left(-\frac{1}{2}\Id + \OK \right) z = g.
\end{equation*}

Next, let us consider the candidate solution $\bar z$ defined as
\begin{equation}
\label{eq:candidate_sol}
    \bar z(x,t) := \begin{cases}
        -2\; g(L,t+L) &x=0\\
        0& x=L.
    \end{cases}
\end{equation}
Then, by \eqref{eq:finalformK}, we have
\begin{equation}
\label{eq:K_and_candidate}
    \OK \bar z := \begin{cases}
        0 & x= 0\\
        g(L,t) & x=L.
    \end{cases}
\end{equation}

Now, the full equation reads
\begin{equation}
\label{eq:expression_K_and_candidate}
    \left(-\frac{1}{2}\Id +\OK \right) \bar z = \begin{cases}
        g(L,t+L) & x=0,\\
        g(L,t) & x=L, 
    \end{cases}
\end{equation}
which, given the assumption \eqref{eq:assumption_Dirichlet_data} results in
\begin{equation}
\label{eq:expression_K_and_candidate_assumption}
    \left(-\frac{1}{2}\Id +\OK \right) \bar z = \begin{cases}
        g(0,t) & x=0,\\
        g(L,t) & x=L.
    \end{cases}
\end{equation}
Hence, $\bar z$ is the unique solution to \eqref{eq:BIEK}.

%%%%%%%%%%%%%%%%%%%%%%%%%%%%%%%%%%%%%%%%%%%%%%%%%%%%%%%%%%%
\subsection{Numerical results}
Our code is implemented in \textsc{Python}. We use exact integration to obtain 
the entries of the Galerkin matrix, while we resort to built-in quadrature 
routines (\textsc{SciPy.integrate.quad)} to compute the right-hand side of the 
system. The resulting linear systems of equations are solved using a built-in 
direct solver (\textsc{NumPy.Linalg.solve)}. 

We point out that it suffices to test the interior problem since the only 
difference between \eqref{eq:BIEK} and \eqref{eq:BIEKex} is the sign of the 
identity, whose discretisation is stable under our assumptions and choice of 
boundary element spaces. We first validate our theoretical results with $\Omega 
:= (0,3)$, $T=6$, and the following exact solutions for \eqref{eq:IBVP}, which 
are also used in \cite{SUZ21}:
\begin{align}
\label{eq:exp1a}
  u_a(x,t) = \begin{cases}
\frac{1}{2}(t-x-2)^3(t-x)^3 & x\leq t \leq 2+x,\\
0 &  \text{else}.
\end{cases}
\end{align} 
 \begin{align}
 \label{eq:exp1b}
  u_b(x,t) = \begin{cases}
\frac{1}{2}\vert \sin(\pi(x-t)) \vert, &\qquad x\leq t,\\
0 & \qquad \text{else}.
\end{cases}
\end{align}

Our Dirichlet data is obtained by taking the lateral trace on either 
\eqref{eq:exp1a} or \eqref{eq:exp1b}. Since in both cases, it satisfies 
assumption \eqref{eq:assumption_Dirichlet_data}, we are able to find an 
expression for the exact solution to the unknown density, following the cue 
of subsection~\ref{ssec:ExRep}.

We measure the corresponding error for each experiment and the condition number 
$\kappa$ for the Galerkin matrix corresponding to $-\frac{1}{2}\Id + \OK$. First 
we consider a set of uniform meshes. 

%%%%%%%%%%%%%%%%%%%%%%%%%%%%%%%%%%%%%%%
\subsubsection{Piecewise constant discretisation}

First we consider the numerical approximation in $L^2(\Sigma)$, considering the 
Galerkin formulation \eqref{eq:BIEKdis1d_pwc}. The results are presented in 
Table~\ref{tab:exp_pwc}, where $N$ denotes the number of elements used to 
discretise the lateral boundary $\Sigma_h$ and \emph{eoc} stands for estimated 
order of convergence. Let $z_a$ and $z_b$ be the exact 
solutions to \eqref{eq:BIEK} with Dirichlet data corresponding to \eqref{eq:exp1a} 
and \eqref{eq:exp1b} respectively. Additionally, we denote the obtained numerical 
solutions by $z^0_{a,h}$ and $z^0_{b,h}$. We observe that for both experiments, 
the expected convergence rate is achieved. Fun fact: the resulting condition 
number corresponds \emph{exactly} with the golden ratio squared.
\begin{table}[h]
\centering
\caption{Numerical results for \eqref{eq:BIEK} in $L^2$ using $\calS^0(\Sigma_h)$. 
Given Dirichlet data corresponding to solutions \eqref{eq:exp1a}, and 
\eqref{eq:exp1b} and when using uniform meshes. 
}\label{tab:exp_pwc}
\medskip
 \begin{tabular}{|c | c c | c c | c |}\hline
$N$ & $\norm{z_a-z^0_{a,h}}{L^2(\Sigma)}$ & eoc & 
$\norm{z_b-z^0_{b,h}}{L^2(\Sigma)}$ & eoc & $\kappa$\\\hline
64 & 8.77E-02 & - & 2.70E-01 & - & 2.62E+00 \\
128 & 4.41E-02 & 0.99 & 1.42E-01 & 0.93 & 2.62E+00 \\
256 & 2.21E-02 & 1.0 & 7.22E-02 & 0.97 & 2.62E+00 \\
512 & 1.10E-02 & 1.0 & 3.65E-02 & 0.99 & 2.62E+00 \\
1024 & 5.52E-03 & 1.0 & 1.83E-02 & 0.99 & 2.62E+00 \\
2048 & 2.76E-03 & 1.0 & 9.18E-03 & 1.0 & 2.62E+00 \\
\hline
 \end{tabular} 
 \end{table} 

%%%%%%%%%%%%%%%%%%%%%%%%%%%%%%%%%%%%%%%
\subsubsection{Piecewise linear discretisation}\label{sssec:Numeric_pwl}
The results related to \eqref{eq:BIEKdis1d_pwl} are displayed in 
Table~\ref{tab:exp1}, where $N$ denotes the number of elements used to 
discretise the lateral boundary $\Sigma_h$. Like in the piecewise constant 
case, we let $z_a$ and $z_b$ be the exact solutions to \eqref{eq:BIEK} with 
Dirichlet data corresponding to \eqref{eq:exp1a} and \eqref{eq:exp1b} 
respectively. Furthermore, we denote the obtained numerical solutions by 
$z^1_{a,h}$ and $z^1_{b,h}$.  Expected convergence rate is achieved, leading 
to higher or equal rate compared to the piecewise constant discretisation.
\begin{table}[h]
\centering
\caption{Numerical results for \eqref{eq:BIEK} in $H^{1/2}_{0,}(\Sigma)$ using 
$\calS^1_{0,}(\Sigma_h)$. Given Dirichlet data corresponding to solutions 
\eqref{eq:exp1a}, and \eqref{eq:exp1b} and when using uniform meshes. 
}\label{tab:exp1}
\medskip
 \begin{tabular}{|c | c c | c c | c |}\hline
$N$ & $\norm{I_h z_a-z^1_{a,h}}{H^{1/2}_{0,}(\Sigma)}$ & eoc & 
$\norm{I_h z_b-z^1_{b,h}}{H^{1/2}_{0,}(\Sigma)}$ & eoc & $\kappa$\\\hline
64   & 4.34E-02 & -    & 3.91E-01 & -    & 8.09E+00\\
128  & 1.47E-02 & 1.57 & 1.88E-01 & 1.06 & 8.12E+00\\
256  & 5.10E-03 & 1.52 & 9.22E-02 & 1.03 & 8.12E+00\\
512  & 1.79E-03 & 1.51 & 4.57E-02 & 1.01 & 8.13E+00\\
1024 & 6.33E-04 & 1.5  & 2.28E-02 & 1.01 & 8.13E+00 \\
2048 & 2.24E-04 & 1.5  & 1.13E-02 & 1.01 & 8.13E+00\\
\hline
 \end{tabular} 
 \end{table} 

Next, we rerun the experiments on a family of meshes which do not meet
Assumption~\ref{Assumption:mesh}.
In order to achieve this, we consider a terminal 
time $T=2\pi$ and $L=3$, ensuring that on each boundary no uniform refinement 
will lead to vertices at intersections between time-slices. Moreover, the 
boundaries at $x=0$ and $x=L$ are divided into $N_0$ and $N_L$ elements 
respectively, where we impose $N_0 \neq N_L$. Consider an initial mesh with 
$N_0=40$ and $N_L= 24$. We create the family of meshes by uniformly refining 
the initial mesh, which is displayed in Figure~\ref{fig:initial_mesh}. 

\begin{figure}[h]
\centering
\begin{tikzpicture}
% [every rectangle node/.style={draw},
% every circle node/.style={draw,double}]

\shade[left color=white,right color=lightgray] (0,0) rectangle (3,3) node[pos=.5] {$T_1$};
\shade[left color=lightgray,right color=white] (0,3) rectangle (3,6) node[pos=.5] {$T_2$};
\shade[left color=white,right color=lightgray] (0,6) rectangle (3,2*pi) node[pos=.5] {};

\def\lastx{-0.05}
\draw[black] (0,0)--(0,2*pi);
\draw[black] (3,0)--(3,2*pi);
\foreach \x[remember =\x as \lastx] in {0,2,...,48}{
    \draw (3,\x*pi/24) node[circle, fill =black, scale = 0.1, draw=black] {};
%     \draw (5,\x+5) node[rectangle, fill =orange, scale = 0.4, draw= black,double] {};
%       \draw[ultra thick, lightgray] (5,\lastx+5.05)--(5,5+\x);
%       \draw[ultra thick, lightgray] (0,\lastx+0.05)--(0,0+\x);
    }
\def\lastx{-0.05}
\foreach \x[remember =\x as \lastx] in {0,2,...,80}{
    \draw (0,\x*pi/40) node[circle, fill =black, scale = 0.1, draw= black] {} ;
    % \draw (3,\x) node[circle, fill =blue, scale = 0.4, draw= black, double] {};
%        \draw[ultra thick, black] (0,\lastx+5.05)--(0,5+\x);
%        \draw[ultra thick, black] (5,\lastx+0.05)--(5,0+\x);
    }
\draw[dotted] (-1.5,0)  -- (0,0) node[pos=-0.5] {$t=0$};
\draw[dotted] (-1.5,3)  -- (0,3) node[pos=-0.5] {$t=L$};
\draw[dotted] (-1.5,6)  -- (0,6) node[pos=-0.5] {$t=2L$};
\draw[dotted] (0,-1)  -- (0,0) node[pos=-0.3] {$\Sigma_h^0$};
\draw[dotted] (3,-1)  -- (3,0) node[pos=-0.3] {$\Sigma_h^L$};

\end{tikzpicture}%
\caption{Initial non-uniform mesh which does not satisfy assumption $T=nL$ for 
some $n\in\N$.}
\label{fig:initial_mesh}
\end{figure}
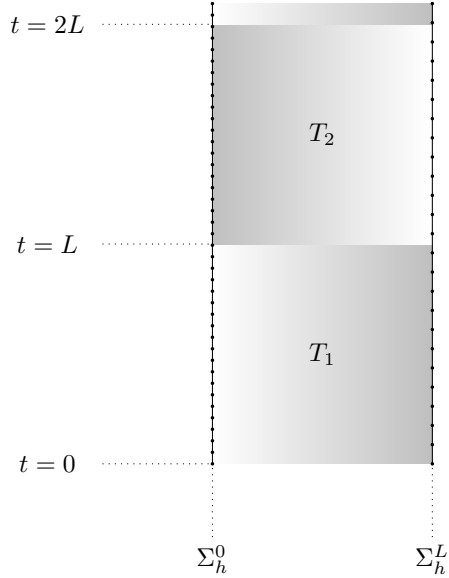

The results are given in Table~\ref{tab:exp_non_uni}. We again observe the 
expected convergence rate in spite of the fact that our meshes violate 
Assumption~\ref{Assumption:mesh}.

\begin{table}[H]
\centering
\caption{Numerical results for \eqref{eq:BIEK} given Dirichlet data 
corresponding to solutions \eqref{eq:exp1a} and \eqref{eq:exp1b}, and using a 
family of meshes not meeting Assumption $T=nL$ for 
some $n\in\N$., with $N=N_0+N_L$.  
}\label{tab:exp_non_uni}
\medskip
 \begin{tabular}{|c | c c | c c | c |}\hline
$N$ & $\norm{I_h z_a-z^1_{a,h}}{H^{1/2}_{0,}(\Sigma)}$ & eoc & 
$\norm{I_h z_b-z^1_{b,h}}{H^{1/2}_{0,}(\Sigma)}$ & eoc & $\kappa$\\\hline
64 & 3.26E-02 & - & 3.41E-01 & - & 9.10E+00 \\
128 & 1.12E-02 & 1.55 & 1.80E-01 & 0.93 & 9.96E+00 \\
256 & 3.90E-03 & 1.52 & 8.47E-02 & 1.08 & 1.03E+01 \\
512 & 1.37E-03 & 1.51 & 3.88E-02 & 1.13 & 9.99E+00 \\
1024 & 4.85E-04 & 1.5 & 2.17E-02 & 0.84 & 1.04E+01 \\
2048 & 1.71E-04 & 1.5 & 1.08E-02 & 1.01 & 1.02E+01 \\
\hline
 \end{tabular} 
 \end{table} 

\subsubsection{Discrete inf-sup constant}
In light of the results from Table~\ref{tab:exp_non_uni}, we want to understand 
if the convergence we observed was ``a lucky strike'' or if we have discrete 
stability regardless of Assumption~\ref{Assumption:mesh}. We do this by checking 
numerically the obtained discrete inf-sup constant and whether it remains 
constant with mesh refinement.

We compute the discrete inf-sup constant following the approach in \cite[Rem.~3.159]{Joh10}. In the case of computing the discrete inf-sup constant in $H^{1/2}_{0,}(\Sigma)$, 
we require the following mass matrix
\begin{equation}
\label{eq:dual_inner}
    M_y[i,j]= \inner{\psi_i}{\psi_j}_{[H^{1/2}_{0,}(\Sigma)]^\prime},
\end{equation}
where $\spn \{ \psi_i\}_{i=1}^N = \calS^1_{0,}(\Sigma_h)$. 

Since piecewise linear basis functions are in $L^2(\Sigma)$, we use the 
following equivalence
\begin{equation}
\label{eq:magic}
\inner{\psi_i}{\psi_j}_{[H^{1/2}_{0,}(\Sigma)]^\prime} \equiv -\dual{\calH_T 
\psi_i}{\overline{\partial}_t^{-1} \psi_j}_{\Sigma},
\end{equation}
where $\overline{\partial}_t^{-1} f(s) := -\int_t^T f(s) \;ds$ 
for $f\in L^2(\Sigma)$. The proof of \eqref{eq:magic} can be found in 
Appendix~\ref{app:dualNorm}.

Let $h_0$ and $h_L$ denote the mesh-width of uniformly refined left ($x=0$) and 
right ($x=L$) boundary respectively. As in Section \ref{sssec:Numeric_pwl}, when 
computing the inf-sup constant, we distinguish between the cases where the 
assumption $T=nL$, with $n\in \N$ does or does not hold. More specifically, we 
consider: i) a uniform mesh with $T=nL$ and $h_0 =h_L$; ii) a (non-uniform) 
mesh with $T=n \pi$ and $h_0 \neq h_L$. 

\begin{figure}[h!]
\centering
\pgfplotstableread{Data/inf_sup_2ndkind_Hhalf_Tfixed_uniform.txt}\infsupUni
\pgfplotstableread{Data/inf_sup_2ndkind_Hhalf_Tfixed_nonuniform.txt}\infsupNon
% \pgfplotstabletranspose\infsupDefault{Data/inf_sup_default.txt}
% \pgfplotstabletranspose\infsupEBEM{Data/inf_sup_EBEM.txt}

% \pgfplotstabletypeset[string type]\infsupDefault

% \pgfplotstabletranspose\infsupUni{Data/inf_sup_2ndkind_Hhalf_Tfixed_uniform.txt}
% \pgfplotstabletranspose\infsupNon{Data/inf_sup_2ndkind_Hhalf_Tfixed_nonuniform.txt}

\pgfplotstableset{
    create on use/Nh/.style={
        create col/copy column from table={Data/N.txt}{N}}
    }
{\scalefont{1.5}
\begin{tikzpicture}
\begin{loglogaxis}[
    log basis y ={2},
    log basis x = {2},
    width=14cm,
    height=14cm,
    xlabel={$N_h$},
    ylabel={Inf-sup constant},
    xmin=64, xmax=1024,
    ymin = 0.1, ymax=0.5,
    ytick pos= left,
    xtick pos= bottom,
    xtick={64,128,256,512,1024},
    xticklabels= {64,128,256,512,1024},
    legend pos=north east,
    legend cell align={left},
    legend style ={font = \large, inner sep= 6pt, outer sep = -6pt},
    ymajorgrids=true,
    grid style=dashed,
]

% \addplot [
%     domain= 1:10, 
%     samples=10, 
%     color=blue,
%     dotted,
%     line width=2pt
% ]
% {(1/2-1/6)*2*sin(deg(pi/(4*x+2)))};
% \addlegendentry{$\gamma_n$}

% \addplot [
%     domain= 1:10, 
%     samples=10, 
%     color=black,
%     dashed,
%     line width=2pt
% ]
% {sin(deg(pi/(4*x+2)))};
% \addlegendentry{$\tilde c_S(n)$}

% \addplot [
%     domain= 32:2048, 
%     samples=8, 
%     color=black,
%     dashed,
%     thick,
%     mark = 10-pointed star,
% ]
% {100*x^-3};
% \addlegendentry{eoc = 3}

% \addplot[orange,mark= square*,line width=2pt, mark options={solid}, dotted] table [x = Nh, y index = 1] {\uni};
% \addlegendentry{$\|p_h\|_{L^2(\Sigma)}$-uni}

% \addplot[green!60!black, mark=square*,mark options=solid, line width=2pt] table [x = Nh, y index = 2] {\uni};
% \addlegendentry{$\|w-w_H\|_{[H_{,0}^1(\Sigma)]^\prime}$-uni}

\addplot[orange,mark= square*,line width=2pt, mark options={solid}] table [x = Nh, y index = 0] {\infsupUni};
\addlegendentry{Inf-sup constant - Uniform}

\addplot[green!60!black, mark=triangle*,mark options=solid, line width=2pt] table [x = Nh, y index = 0] {\infsupNon};
\addlegendentry{Inf-sup constant - Non-Uniform}

\end{loglogaxis}
\end{tikzpicture}
}%
\caption{Discrete inf sup constant for different mesh refinements and two types 
of meshes: the uniform case with $T=6$; 
the non-uniform case with $T=2\pi$ and initial mesh given by Figure 
\ref{fig:initial_mesh}.}
\label{fig:inf_sup_fixedT}
\end{figure}
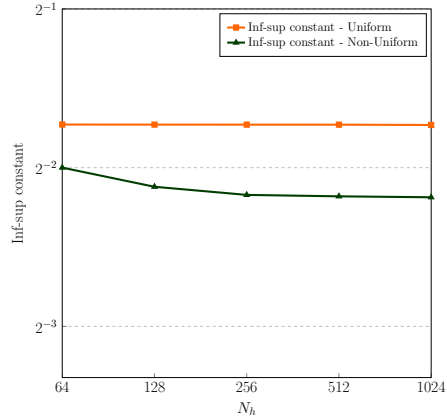

Figure~\ref{fig:inf_sup_fixedT} displays the discrete inf-sup constant against 
number of mesh elements. There, we set the final time for the uniform case to 
$T=6$, and in the non-uniform case to $T=2\pi$, and increase the level of 
refinement. We see how the discrete inf-sup constant is asymptotically 
independent of the mesh-width, which corroborates that the formulation is 
uniformly discrete inf-sup stable. 

\begin{figure}[h!]
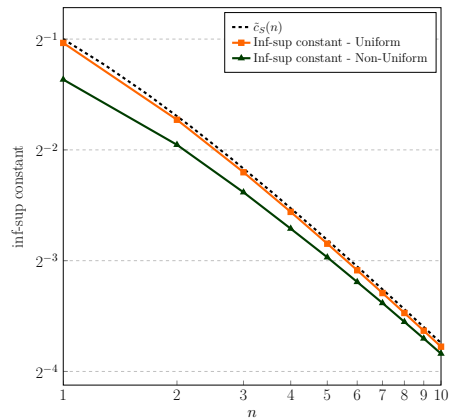

\centering
\includestandalone[width=0.5\textwidth]{images/inf_sup_Tdiff}%
\caption{Discrete inf sup constant for different final times: the uniform case 
with $T=nL$ and $h_0=h_L=3/16$; the non-uniform case with $T=n \pi$ and $h_0 = 
\pi/20$ and $h_L = \pi/12$. Here $n$ is the number of timeslices (and thus 
refinement independent).}
\label{fig:inf_sup_timeslices}
\end{figure}

Next, we study numerically how the discrete inf-sup constant depends on the 
final time $T$. As shown in Figure \ref{fig:inf_sup_timeslices} the discrete 
inf-sup constant decreases when the terminal time $T$ increases, yet it remains 
strictly positive. In case i), the discrete inf-sup constant asymptotically 
corresponds to the inf-sup constant $\tilde c_s(n)=\sin \frac{\pi}{2(2n+1)}$.

%%%%%%%%%%%%%%%%%%%%%%%%%%%%%%%%%%%%%%%%%%%%%%%%%%%%%%%%%%%%%%%%%%%%%%%%%%%%%%%
\section{Conclusion}
We presented a simple second-kind formulation for the transient wave equation 
in $1d$ that has a proven stable low-order discretisations. We point 
out that this is the first time that such a result is available without having 
to resort to a combined field boundary integral operator. It is also a required 
and insightful first step to achieve the same result for higher 
spatial-dimensions.

For $L^2$ the stability will hold for any conforming discretisation, while 
for $H^{1/2}_{0,}$ we found stable inf-sup pairs. For the proof we required 
Assumption~\ref{Assumption:mesh}, yet we observe in the numerics that this 
is not needed. Ongoing work is showing stability without this mesh 
assumption.

It is important to keep in mind that more work is required to generalise 
the obtained discrete stability to $\Omega\in\R^2$ with $d=2,3$.
In particular, one needs 
to extend the auxiliary Lemma~\ref{lem:Qh_stab_dualspace} to the space 
$\calH_{0,}(\Sigma)$ defined in \cite{StU22}, since $\OK : \calH_{0,}(\Sigma) 
\to \calH_{0,}(\Sigma)$ in higher dimensions; and to study the mapping 
properties of the double layer operator $\OK$, such that the related projections 
remain meaningful. The latter result is dimension dependent as a consequence 
of the fundamental solution being so. For the auxiliary Lemma, however, the 
density of $H^{1/2}_{0,}(\Sigma)$ in $\calH_{0,}(\Sigma)$ is a promising 
starting point.

We are also cautiously optimistic about the extension of Theorem~\ref{thm:L2ellip} 
to higher spatial-dimensions. The reason is that, in its proof, one can see how 
the proposed second kind formulation is 
connected with energetic BEM for the first kind formulation arising from the 
weakly singular operator $\OV$. Indeed, the inf-sup constant $c_s(n)$ agrees 
with the $L^2$-ellipticity constant from \cite[Theorem~2.1]{ADG09}. Hence, the 
fact that energetic BEM is known to work in the case of two spatial-dimensions, 
makes us believe that there is a chance that we can work out the required theory.

Ongoing and future work involves further understanding of the trace spaces 
$\calH_{0,}(\Sigma)$ and the BEM implementation in higher dimensions.

%%%%%%%%%%%%%%%%%%%%%%%%%%%%%%%%%%%%%%%%%%%%%%%%%%%%%%%%%%%%%%%%%%%%%%%%%%%%%%%
\bibliographystyle{plain}
\bibliography{references}

\appendix
%%%%%%%%%%%%%%%%%%%%%%%%%%%%%%%%%%%%%%%%%%%%%%%%%%%%%%%%%%%%%%%%%%%%%%%%%%%%%%%
\section{Representation of the Dual Inner Product}
\label{app:dualNorm}
First we recall \cite[p.~9]{SUZ21} for $u,v \in H^{1/2}_{0,}(\Sigma)$
\begin{equation}\label{eq:mht_inner}
    \dual{\calH_T u}{\partial_t v}_\Sigma = \inner{u}{v}_{H^{1/2}_{0,}(\Sigma)}.
\end{equation}

The fact that $\partial_t: H^{1}_{,0}(\Sigma) \to  L^2(\Sigma)$ is an isomorphism 
\cite[p.~6]{SUZ21}, implies the existence of an inverse operator 
$\overline{\partial}_t^{-1}:L^2(\Sigma) \to H^{1}_{,0}(\Sigma)$ which also is 
an isomorphism. In this case, the inverse operator can be expressed as \cite[p.~6]{SUZ21}
\begin{equation}
    \overline{\partial}_t^{-1} f(s) := -\int_t^T f(s) \;ds, \quad 
    \text{ for } f \in L^2(\Sigma).
\end{equation}
Similarly, the isomorphism $\partial_t: H^{1}_{0,}(\Sigma) \to L^2(\Sigma)$ 
has $\partial_t^{-1}:L^2(\Sigma) \to H^{1}_{0,}(\Sigma)$ as its inverse 
operator, which is given by
\begin{equation}
\partial_t^{-1} f(s) := -\int_0^t f(s) \;ds, \quad 
\text{ for } f \in L^2(\Sigma).
\end{equation}

Next, we study the properties of $\overline{\partial}_t^{-1}$ properties in the 
fractional Sobolev spaces under consideration. For this, we need to consider the 
time derivative in the weak sense.

\begin{lemma}
\label{lem:repOmega}
  Let $\omega\in [H^1_{0,}(\Sigma)]^\prime $. Then there exists a unique $v\in L^2(\Sigma)$
  such that 
  \begin{equation}\label{eq:w-partialv}
    -(v,\partial_t u)_{L^2(\Sigma)} =
   \dual{\omega}{u}_{\Sigma} \qquad
    \forall u\in H^1_{0,}(\Sigma),
  \end{equation}
  and
  \begin{equation}\label{eq: representation w}
    \|\omega\|^2_{[H^1_{0,}(\Sigma)]^\prime } = \|v\|_{L^2(\Sigma)}^2.
  \end{equation}
\end{lemma}
\begin{proof}
First, one can show that \eqref{eq:w-partialv} admits a unique solution 
$v\in L^2(0,T)$ by using the Babu\u{s}ka--Ne\v{c}as theory. This implies 
that $\partial_t : L^2(0,T) \to [H^1_{0,}(0,T)]^\prime $ is an isomorphism.
Then, we can derive the following
\begin{eqnarray*}
\norm{\omega}{[H^1_{0,}(0,T)]^\prime }
& = & \sup_{u\in H^1_{0,}(0,T)\setminus\lbrace0\rbrace}
    \frac{\abs{\dual{\omega}{u}_{(0,T)}}}{\norm{\partial_t u}{L^2(0,T)}} \\
& = & \sup_{u\in H^1_{0,}(0,T)\setminus\lbrace0\rbrace}
        \frac{\abs{-(v,\partial_t u)_{L^2(\Sigma)}}}
    {\norm{\partial_t u}{L^2(0,T)}} \leq\norm{v}{L^2(0,T)}.
  \end{eqnarray*}
  
Next, let us, consider $\bar u(t) = \int_0^t v(s)\, ds$, which satisfies 
$\partial_t \bar u = v$ in $L^2(0,T)$. With this, we arrive to 
\begin{equation*}
\norm{v}{L^2(0,T)} =
\frac{\abs{-(v,\partial_t \bar u)_{L^2(\Sigma)}}}
{\norm{\partial_t \bar u}{L^2(0,T)}} \leq
\sup_{u\in H^1_{0,}(0,T)\setminus\lbrace0\rbrace}
        \frac{\abs{-(v,\partial_t u)_{L^2(\Sigma)}}}
    {\norm{\partial_t u}{L^2(0,T)}} = \norm{\omega}{[H^1_{0,}(0,T)]^\prime }.
\end{equation*}
\end{proof}

Lemma~\ref{lem:repOmega} implies the existence of the inverse of the 
(weak) time derivative, i.e., $\overline{\partial}_t^{-1} : [H^1_{0,}(0,T)
]^\prime 
\to L^2(0,T)$, which is also an isomorphism. As discussed earlier, we also have 
$\overline{\partial}_t^{-1}:L^2(\Sigma) \to H^{1}_{,0}(\Sigma)$. Hence, by 
interpolation, we conclude that $\overline{\partial}_t^{-1} : [H^{1/2}_{0,}(0,T
)]^\prime \to H^{1/2}_{,0}(\Sigma)$ is an isomorphism as well. Moreover, it is 
also an isometry between these spaces.

Another important ingredient in this Appendix is the modified Hilbert transform 
$\calH_T$ introduced in \cite[Section~2.4]{StZ20}. We point out  that it can be 
seen as a mapping $\calH_T: H^{1}_{0,}(\Sigma)\to H^{1}_{,0}(\Sigma)$ or as a 
mapping $\calH_T: H^{1/2}_{0,}(\Sigma)\to H^{1/2}_{,0}(\Sigma)$, which satisfy
\begin{equation}
 \norm{u}{H^{1}_{0,}(\Sigma)}=\norm{\calH_T u}{H^{1}_{,0}(\Sigma)}, 
 \quad \text{ and } \quad 
 \norm{u}{H^{1/2}_{0,}(\Sigma)}=\norm{\calH_T u}{H^{1/2}_{,0}(\Sigma)}.
\end{equation}
Similarly, we have the inverse modified Hilbert transform that induces the 
mappings $\calH_T^{-1}: H^{1}_{,0}(\Sigma)\to H^{1}_{0,}(\Sigma)$ and 
$\calH_T^{-1}: H^{1/2}_{,0}(\Sigma)\to H^{1/2}_{0,}(\Sigma)$, which verify
\begin{equation}
 \norm{u}{H^{1}_{,0}(\Sigma)}=\norm{\calH_T^{-1} u}{H^{1}_{0,}(\Sigma)}, 
 \quad \text{ and } \quad 
 \norm{u}{H^{1/2}_{,0}(\Sigma)}=\norm{\calH_T^{-1} u}{H^{1/2}_{0,}(\Sigma)}.
\end{equation}

Now we are in the position to state the main result of this Appendix:
\begin{proposition}
\label{prop:dualNorm}
The following representation of the $[H^{1/2}_{0,}(\Sigma)
]^\prime$ inner product holds:
\begin{equation}\label{eq:equivalent_inner}
    \inner{\omega}{\eta}_{[H^{1/2}_{0,}(\Sigma)]^\prime} \equiv 
%     -\dual{\calH_T u}{\overline{\partial}_t^{-1} v}
    \dual{\overline{\partial}_t^{-1} \omega}{\overline{\partial}_t \calH_T^{-1}
    \partial_t^{-1}\eta}_\Sigma 
\end{equation}
for $\omega,\eta \in [H^{1/2}_{0,}(\Sigma)]^\prime$.
Moreover, when $u,v \in L^2(\Sigma)$ (as piecewise linear functions do), we can further simplify this expression to
\begin{equation}\label{eq:equivalent_innerL2}
    \inner{u}{v}_{[H^{1/2}_{0,}(\Sigma)]^\prime} \equiv 
    -\dual{\calH_T u}{\overline{\partial}_t^{-1} v}_\Sigma.
\end{equation}
\end{proposition}

Next, we introduce some intermediate results in order to proceed to prove 
Proposition~\ref{prop:dualNorm}. We begin by the following auxiliary lemma, 
which is closely related to \cite[Lem.~2.3]{StZ20}.
\begin{lemma}
\label{lem:newlem23}
For $u\in H^{1/2}_{,0}(\Sigma)$ we have
\begin{equation*}
    \dual{\partial_t \calH_T^{-1}u}{v}_\Sigma = -\dual{\calH_T 
    \partial_t u}{v}_\Sigma, \quad \text{for all }v\in H^{1/2}_{,0}(\Sigma).
\end{equation*}
\end{lemma}
\proof{The proof works analogously to Lemma~2.3 in \cite{StZ20} but we include 
it for completeness. Consider a function $\varphi \in C^\infty(0,T)$ with 
$\varphi(T)=0$. We remind the reader that it can be expanded as follows
\begin{equation*}
\varphi(t) = \sum_{k=0}^\infty \varphi_k \cos\left[ \left(\frac{\pi}{2}
+k\pi\right)\frac{t}{T}\right],
\end{equation*}
where 
\begin{equation*}
\varphi_k := \frac{2}{T}\int_0^T \varphi(t) \cos\left[ \left(\frac{\pi}{2}+k\pi
\right)\frac{t}{T}\right]\; dt.
\end{equation*}
By definition of the inverse 
modified Hilbert transform $\calH_T^{-1}$, we have that
\begin{equation*}
\calH_T^{-1} \varphi(t) = \sum_{k=0}^\infty \varphi_k \sin\left[ \left(\frac{
\pi}{2}+k\pi\right)\frac{t}{T}\right],
\end{equation*}

From the above expression it becomes clear that
\begin{equation}
\partial_t \calH_T^{-1} \varphi(t) = \frac{1}{T}\sum_{k=0}^\infty \varphi_k 
\cos\left[ \left(\frac{\pi}{2}+k\pi\right)\frac{t}{T}\right] 
\left(\frac{\pi}{2}+k\pi\right).
\end{equation}

Additionally, we have that 
\begin{equation}
\partial_t \varphi = -\frac{1}{T}\sum_{k=0}^\infty \varphi_k \sin\left[
\left(\frac{\pi}{2}+k\pi\right)\frac{t}{T}\right] \left(\frac{\pi}{2}+k\pi
\right).
\end{equation}
Hence, by the definition of the modified Hilbert transform, we get
\begin{equation}
-\calH_T \partial_t \varphi = \frac{1}{T}\sum_{k=0}^\infty u_k \cos\left[ 
\left(\frac{\pi}{2}+k\pi\right)\frac{t}{T}\right] \left(\frac{\pi}{2}
+k\pi\right) = \partial_t \calH_T^{-1} \varphi,
\end{equation}
for all $\varphi \in C^\infty(0,T)$ with $\varphi(T)=0$. Finally,
the result follows by completion.\\\qed}

Similarly, we can prove
\begin{lemma}
\label{lem:HTder}
For $u\in H^{1}_{,0}(\Sigma)$ we have
\begin{equation*}
    \dual{\partial_t \calH_T^{-1}u}{v}_\Sigma = -\dual{\calH_T 
    \partial_t u}{v}_\Sigma, \quad \text{for all }v\in H^{1}_{,0}(\Sigma).
\end{equation*}
\end{lemma}

Finally, we conclude this Appendix by proving its main result.
\proof[Proof of Proposition~\ref{prop:dualNorm}]{
We proceed by showing that: (i) the introduced bilinear forms can be used to 
represent the $[H^{1/2}_{0,}(\Sigma)]^\prime$-norm (i.e., that they are 
continuous and elliptic); and (ii) they are symmetric. With these two pieces, 
and by using the polarization identity on Hilbert spaces, we conclude the 
desired inner-product representations.\\

\noindent\underline{\emph{Step 1: The bilinear form induces the $[H^{1/2}_{0,}(\Sigma)]^\prime$ norm.}}\\ 
First, we point out that, by composition, $\calH_T^{-1}
\overline{\partial}_t^{-1}: [H^{1/2}_{0,}(\Sigma)]^\prime \to H^{1/2}_{0,}(\Sigma)$ 
is an isometric isomorphism, hence 
\begin{equation}
\norm{\calH_T^{-1}\overline{\partial}_t^{-1}\omega}{H^{1/2}_{0,}(\Sigma)}
= \norm{\omega}{ [H^{1/2}_{0,}(\Sigma)]^\prime}.
\end{equation}

Furthermore, by \eqref{eq:mht_inner}, we have that for $\omega \in [H^{1/2}_{0,
}(\Sigma)]^\prime$
\begin{align*} 
\norm{\calH_T^{-1}\overline{\partial}_t^{-1}\omega}{ H^{1/2}_{0,}(\Sigma) }^2 
&= \dual{\calH_T \calH_T^{-1}\overline{\partial}_t^{-1} \omega}{\partial_t 
\calH_T^{-1}\overline{\partial}_t^{-1} \omega}_\Sigma\\ &= 
\dual{\overline{\partial}_t^{-1} \omega}{\partial_t \calH_T^{-1}\overline{
\partial}_t^{-1}\omega}_\Sigma, 
\end{align*}

which proves that the bilinear form on the right hand side of 
\eqref{eq:equivalent_inner} does indeed represent the $[H^{1/2}_{0,}(\Sigma)
]^\prime$-norm. 

Finally, for $u \in L^2(\Sigma)$, we can use Lemma~\ref{lem:newlem23} and get
\begin{align*}
\norm{\calH_T^{-1}\overline{\partial}_t^{-1}u}{ H^{1/2}_{0,}(\Sigma) }^2 &=
\dual{\overline{\partial}_t^{-1} u}{\partial_t \calH_T^{-1}\overline{\partial
}_t^{-1}u}_\Sigma = -\dual{\overline{\partial}_t^{-1} u}{\calH_T u}_\Sigma,
\end{align*}
from where we conclude that the bilinear form on the right hand side of 
\eqref{eq:equivalent_innerL2} induces the $[H^{1/2}_{0,}(\Sigma)]^\prime$-norm.\\

\noindent\underline{\emph{Step 2: The obtained bilinear forms are symmetric.}}\\ 
From \cite[Corollary~2.5]{StZ20}, we know that
\begin{align*}
\dual{\calH_T w}{\partial_t z}_\Sigma = \dual{\partial_t w}{\calH_T z}_\Sigma, 
\quad \forall w,z \in H^{1/2}_{0,}(\Sigma).
\end{align*}
Hence
\begin{align*}
\dual{\overline{\partial}_t^{-1} \omega}{\partial_t \calH_T^{-1}\overline{
\partial}_t^{-1}\eta}_\Sigma &=
\dual{\calH_T (\calH_T^{-1}\overline{\partial}_t^{-1} \omega)}{\partial_t 
(\calH_T^{-1}\overline{\partial}_t^{-1} \eta)}_\Sigma \\
&=  \dual{\partial_t (\calH_T^{-1}\overline{\partial}_t^{-1} \omega)}{\calH_T 
(\calH_T^{-1}\overline{\partial}_t^{-1} \eta)}_\Sigma\\
&= \dual{\partial_t \calH_T^{-1}\overline{\partial}_t^{-1} \omega}{\overline{
\partial}_t^{-1} \eta}_\Sigma,
\end{align*}
for $\omega, \eta \in [H^{1/2}_{0,}(\Sigma)]^\prime$, and 
\begin{align*}
-\dual{\overline{\partial}_t^{-1} u}{\calH_T v}_\Sigma = -\dual{\calH_T u}{\overline{\partial}_t^{-1} v}_\Sigma, \quad \forall u,v \in L^2(\Sigma).
\end{align*}
\qed
}

\end{document}

\pgfplotstableread{Data/inf_sup_2ndkind_Hhalf_Tdiff_uniform.txt}\infsupUni
\pgfplotstableread{Data/inf_sup_2ndkind_Hhalf_Tdiff_nonuniform.txt}\infsupNon
% \pgfplotstabletranspose\infsupDefault{Data/inf_sup_default.txt}
% \pgfplotstabletranspose\infsupEBEM{Data/inf_sup_EBEM.txt}

% \pgfplotstabletypeset[string type]\infsupDefault

% \pgfplotstableset{
%     create on use/Nh/.style={
%         create col/copy column from table={Data/N.txt}{N}}
%     }
{\scalefont{1.5}
\begin{tikzpicture}
\begin{loglogaxis}[
    log basis y ={2},
    log basis x = {2},
    width=14cm,
    height=14cm,
    xlabel={$n$},
    ylabel={inf-sup constant},
    xmin=1, xmax=10,
    ytick pos= left,
    xtick pos= bottom,
    xtick={1,2,3,4,5,6,7,8,9,10},
    xticklabels= {1,2,3,4,5,6,7,8,9,10},
    ytick = {0.5,0.25,0.125, 0.0625},
    legend pos=north east,
    legend cell align={left},
    legend style ={font = \large, inner sep= 6pt, outer sep = -6pt},
    ymajorgrids=true,
    grid style=dashed,
]

% \addplot [
%     domain= 1:10, 
%     samples=10, 
%     color=blue,
%     dotted,
%     line width=2pt
% ]
% {(1/2-1/6)*2*sin(deg(pi/(4*x+2)))};
% \addlegendentry{$\gamma_n$}

\addplot [
    domain= 1:10, 
    samples=10, 
    color=black,
    dashed,
    line width=2pt
]
{sin(deg(pi/(4*x+2)))};
\addlegendentry{$\tilde c_S(n)$}

% \addplot [
%     domain= 32:2048, 
%     samples=8, 
%     color=black,
%     dashed,
%     thick,
%     mark = 10-pointed star,
% ]
% {100*x^-3};
% \addlegendentry{eoc = 3}

% \addplot[orange,mark= square*,line width=2pt, mark options={solid}, dotted] table [x = Nh, y index = 1] {\uni};
% \addlegendentry{$\|p_h\|_{L^2(\Sigma)}$-uni}

% \addplot[green!60!black, mark=square*,mark options=solid, line width=2pt] table [x = Nh, y index = 2] {\uni};
% \addlegendentry{$\|w-w_H\|_{[H_{,0}^1(\Sigma)]^\prime}$-uni}

\addplot[orange,mark= square*,line width=2pt, mark options={solid}] table [x expr=\coordindex+1, y index = 0] {\infsupUni};
\addlegendentry{Inf-sup constant - Uniform}

\addplot[green!60!black, mark=triangle*,mark options=solid, line width=2pt] table [x expr = \coordindex+1, y index = 0] {\infsupNon};
\addlegendentry{Inf-sup constant - Non-Uniform}

\end{loglogaxis}
\end{tikzpicture}
}